\documentclass[english]{amsart}
\usepackage[T1]{fontenc}
\usepackage[latin9]{inputenc}
\usepackage{amsmath}
\usepackage{amssymb}
\usepackage{babel}
\usepackage{mathrsfs}
\newtheorem{theorem}{Theorem}[section]
\newtheorem{thm}[theorem]{Theorem}

\newtheorem{lem}[theorem]{Lemma}

\newtheorem{pro}[theorem]{Proposition}
 
\newtheorem{cor}[theorem]{Corollary} 
\newtheorem{example}[theorem]{Example}

\newtheorem{defn}[theorem]{Definition} 
\theoremstyle{remark}
\newtheorem{remark}[theorem]{Remark}
\begin{document}
\title{Equationally defined classes of semigroups}
\author{Peter M. Higgins}
\address{Department of Mathematical Sciences, University of Essex, UK}
\email{peteh@essex.ac.uk}
\author{Marcel Jackson}
\address{Department of Mathematical and Physical Sciences, La Trobe University, VIC 3086, Australia}
\email{m.g.jackson@latrobe.edu.au}
%\date{{}}
\begin{abstract}
We apply, in the context of semigroups, the main theorem from~\cite{higjac}
that an elementary class $\mathcal{C}$ of algebras which is closed under
the taking of direct products and homomorphic images is defined by
systems of equations. We prove a dual to the Birkhoff theorem in that
if the class is also closed under the taking of containing semigroups,
some basis of equations of $\mathcal{C}$ is free of the $\forall$ quantifier.
  We also observe the decidability of the class of equation systems satisfied by semigroups, via a link to systems of rationally constrained equations on free semigroups. Examples are given of EHP-classes for which neither $(\forall \cdots)(\exists \cdots )$ equation systems nor $(\exists \cdots)(\forall \cdots)$ systems suffice. 
\end{abstract}
\maketitle

\section{Introduction}\label{sec:intro}

A theorem of universal algebra and model theory is proved in \cite{higjac}
that is particularly pertinent to the study of semigroups. The theorem
clarifies why it is so often the case that classes of algebras which
are closed under the taking of arbitrary direct products (P) and homomorphic
images (H), but not necessarily subalgebras~(S), may be defined by
a set of equation-like sentences. The members of these bases are referred to in
\cite{higjac} as \emph{equation systems} and generally involve simultaneous
equations, meaning that the equations are not assumed to be independent
in that a symbol may occur in more than one equation.   More formally,  a semigroup equation system is a quantified conjunction of equalities between semigroup words (elements of a free semigroup).  Thus in comparison with the more familiar theory of varieties and equational logic, we allow both $\forall$ and $\exists$ as quantifiers at the front of the sentence and we allow the logical connective of conjunction $\wedge$ (``and'').  %Constants are allowed only inasmuch as they are in the signature.

For instance, Example 2.1(ii) of~\cite{higjac} shows that Completely regular semigroups
($\mathcal{CR}$) is the class consisting of all semigroups in which the following  
 equation pair may be solved: 
\begin{equation}
(\forall\,a)(\exists\,x)\colon  (a=axa) \wedge (ax = xa).\label{eq:cr}
\end{equation}

These equations are simultaneous and indeed both the parameter $a$ and the variable $x$ feature in each.  
Together they capture the property that each element has an inverse with which it commutes, which is one definition of completely regular semigroups. 
However, we shall show that this class may be defined by a single quantified equation (without conjunction), and in more than one way. 

Classes defined by equation systems in this fashion are referred to as $\{E,H,P\}$-classes or simply $EHP$-classes, with
the $E$ symbol standing for the taking of \emph{elementary substructures}, $H$ for \emph{homomorphic images} and $P$ for \emph{direct products}.   An \emph{elementary
class} (or \emph{first order class}) is one defined by a collection
of first order formulae: the quantifiers refer to elements of algebras, hence ``elementary'', in contrast with second order logic where we may quantify over relations.  The inclusion of the class operator $E$ is necessary to capture equivalence with definability by equation systems, but the reader will not need familiarity with the definition of elementary substructures in this article, because we always proceed by finding equation systems that capture the classes we explore.
It is not currently known if closure under
the operator $EHP$ is enough to ensure closure under the set of operators
$\{E,H,P\}$. 

For relevant background and notation on universal algebra see \cite{ber}, \cite{bursan} (Chapter~V there also serves as a useful first introduction to model theoretic notions for those wanting to delve further), while
 \cite{clipreI},  \cite{hig} and \cite{how} are texts covering semigroup facts and terminology.

In the following theorem, closure under taking elementary substructures is subsumed by the assumption that $\mathcal{C}$ is an elementary class (all elementary classes are closed under the taking of elementary substructures).
\begin{thm}\label{thm:EHP} \cite[Theorem 3.1]{higjac} An elementary class $\mathcal{C}$
equals the class of models of some family of equation systems if and
only if $\mathcal{C}$ is closed under taking homomorphic images of direct
products. If the elementary class is the model class of a single sentence,
then it is a class of models of a single equation system.
\end{thm}

The theorem was inspired by the observation that so many of the fundamental
classes of algebraic semigroups are EHP-classes, but are not varieties,
which is to say the class is \emph{not }closed under the taking of
subsemigroups, and so cannot be defined by semigroup identities. Indeed
since equation systems are examples of first order sentences, a class
defined by the satisfaction of a set of equations is automatically
an elementary class and the initial condition of the theorem is, in
the (easy) forward direction, redundant.

In Section 2 we highlight examples showing just how rich is EHP
theory in the context of semigroups. Many extensively studied semigroup
classes are captured through our approach, in many cases by a single
equation. We verify this in cases where the equational bases differ
from those given in \cite{higjac}. 

An EHP-class $\mathcal{C}$ that is closed under the taking of all embeddings
(and not only elementary embeddings) is a variety in which case, by
Birkhoff's theorem, the equations defining $\mathcal{C}$ may be taken
to be identities, which is to say the equations do not involve the
quantifier $\exists$. It is natural then to consider a kind of dual
to the Birkhoff theorem for EHP-classes $\mathcal{C}$ defined by equations
that are free from the quantifier $\forall$. Clearly such a class
is closed under the taking of containing algebras. In Section 3 we
prove the converse does indeed hold for the class of semigroups.

In Theorem~\ref{thm:EHP}, the number of alternations between the $\forall$
symbol, which qualifies \emph{parameters}, (denoted by lower case
letters from the beginning of the alphabet $a,b,c,\dots$), and the
$\exists$ symbol, which qualifies \emph{variables}, (denoted typically
by $x,y,z$) is finite but unbounded in length. In Section 4 we give
examples of EHP-classes  which have no basis comprised of equations of the form $(\forall \cdots)(\exists \cdots)$ nor the form $(\exists \cdots)(\forall \cdots)$. 
\section{Classical semigroup collections as EHP-classes}\label{sec:classical}

\subsection{Classes of regular semigroups}\label{subsec:reg}

We begin with six examples that highlight how easily important classes
of regular semigroups may be characterised by a single equation. Moreover
the proofs of this are simple but elegant exercises in semigroup theory.
In particular we see in Proposition \ref{pro:reg} how the classes in question
may be characterised by small adjustments to the equation that defines
regularity. 

\begin{pro}\label{pro:reg}
The classes of Regular semigroups ($\mathcal{R}eg$),
Left groups ($\mathcal{LG}$), Right groups ($\mathcal{RG}$), Groups ($\mathcal{G}$),
Completely regular semigroups ($\mathcal{CR}$), and Completely simple
semigroups ($\mathcal{CS}$) are the EHP-classes of semigroups defined
by the following equations. 
\begin{equation}
\mathcal{R}eg:(\forall a)(\exists x)\colon a=axa.\label{pro:reg1}
\end{equation}
\begin{equation}
\mathcal{LG}:(\forall a,b)(\exists x)\colon a=axb,\,\,\mathcal{RG}:(\forall a,b)(\exists x)\colon a=bxa.\label{pro:reg2}
\end{equation}
\begin{equation}
\mathcal{G}:(\forall a,b)(\exists x)\colon a=bxb.\label{pro:reg3}
\end{equation}
\begin{equation}
\mathcal{CR}:(\forall\,a)(\exists x)\colon a=a^{2}xa^{2}.\label{pro:reg4}
\end{equation}
\begin{equation}
\mathcal{CS}:(\forall a,b)(\exists x)\colon  a=abxba.\label{pro:reg5}
\end{equation}
\end{pro}
\begin{proof} The equation \eqref{pro:reg1} may be satisfied in any regular
semigroup $S$ by taking $x\in V(a)$. Conversely, if $x$ satisfies
\eqref{pro:reg1}, then $xax\in V(a)$. 

For equation \eqref{pro:reg2}, let $S$ be a left group and let $a,b\in S$. Since $S$ is left
simple, there exists $y\in S$ such that $a=yb$. By regularity there
exists $z\in S$ such that $a=aza$, whence $a=azyb$. Putting $x=zy$
gives $a=axb$, as required. 

Conversely suppose that $a=axb$ is solvable in a semigroup $S$.
Putting $b=a$ gives $a=axa$ is solvable so that $S$ is regular.
Moreover $a=axb$ implies that $a\leq_{\mathscr{L}}b$, and since $a,b$
are arbitrary it follows that $a\mathrel{\mathscr{L}}b$. Therefore $S$ is left
simple and regular and so $S$ is a left group. The left-right dual
argument shows that $\mathcal{RG}$ is defined by the equation $a=bxa$ in like manner.

For equation \eqref{pro:reg3}, if $G$ is a group then for given $a,b\in G$ there is a unique
solution to \eqref{pro:reg3}, that being $x=b^{-1}ab^{-1}$. 

Conversely, let $S$ be a semigroup in which \eqref{pro:reg3} is solvable. The
equation implies that $H_{a}\leq H_{b}$. By interchanging $a$ and
$b$ we obtain the reverse inequality, whence $S$ consists of a single
$\mathscr{H}$-class, and is therefore a group. 

For equation \eqref{pro:reg4}: from the equation, it follows that $a\mathrel{\mathscr{H}}a^{2}$, from which
we infer that every $\mathscr{H}$-class is a group, whence $S$ is a
union of groups, which is to say that $S$ is completely regular. 

Conversely, let $S$ be a completely regular semigroup. Let $a\in S$
and put $x=b^{3}$, where $b$ is the inverse of $a$ in the group
$H_{a}$. Then $ab=ba$ is the identity element of $H_{a}$ and so
\[
a^{2}xa^{2}=a^{2}b^{3}a^{2}=a(ab)b(ba)a=aba=a,
\]
in accord with \eqref{pro:reg4}. 

Finally, let $S$ be a semigroup that satisfies \eqref{pro:reg5}. By putting $b=a$
we see that~\eqref{pro:reg5} implies \eqref{pro:reg4}, so that $S$ is completely regular. For
any $a,b\in S$, \eqref{pro:reg5} implies that $J_{a}\leq_{\mathscr{J}}J_{b}$, and
by role reversal of $a$ and $b$, the reverse inequality follows
so that that $J_{a}=J_{b}$. Therefore $S$ is a simple completely
regular semigroup, which is to say that $S$ is completely simple. 

Conversely let $S$ be a completely simple semigroup and take $a,b\in S$.
Then we have $a\mathrel{\mathscr{R}}ab\mathrel{\mathscr{L}}b\mathrel{\mathscr{R}}ba\mathrel{\mathscr{L}}a$.
By Green's Lemma, the mapping $\rho_{ba}:H_{a}\rightarrow H_{a}$
whereby $x\rho_{ba}=xba$ is a bijection, as is the the mapping $\lambda_{ab}:H_{a}\rightarrow H_{a}$
whereby $x\lambda_{ab}=abx$. It follows that $\phi=\lambda_{ab}\rho_{ba}=\rho_{ba}\lambda_{ab}:H_{a}\rightarrow H_{a}$
is also a bijection, whereupon there exists a unique $x\in H_{a}$
such that $x\lambda_{ab}\rho_{ba}=abxba=a$, thereby proving that
$S$ satisfies equation \eqref{pro:reg5}.
\end{proof}

Other classes of regular semigroups may be defined by an equation
system consisting of \eqref{pro:reg1} together with one more equation. For instance,
Example 2.1(iii) of~\cite{higjac} shows that \emph{Semilattices of groups}
($\mathcal{SG}$) is the EHP-class of regular semigroups defined by the
additional equation:
\begin{equation}
(\forall\,a,b)(\exists\,y)\colon  ab=bya.\label{eq:abbya}
\end{equation}

Some standard properties used in the description of classes of semigroups
may be expressed by equations, and that allows for abbreviation. For
example, that a certain product $u$ of some parameters and variables
is idempotent we write as $u\in E$, or if $v$ is an inverse of $u$
we write $v\in V(u)$. Properties defined by Green's relations are
generally not intrinsically equational within the class of semigroups
but may become so in the presence of the regularity equation \eqref{pro:reg1}.
However the respective properties of being $\mathscr{G}$-simple for
any of the five Green's relations $\mathscr{G}$ defines an EHP-class,
except for the case of $\mathscr{D}$. In general, bisimple semigroups
are not closed under the taking of direct products. Within $\mathcal{R}eg$,
the condition $u\mathrel{\mathscr{H}}v$ is expressible through equations in
$S$, with similar comments applying to $\mathscr{L},\mathscr{R},$ $\mathscr{J}$,
and indeed $\mathscr{D}$. In particular, the class of regular bisimple
semigroups is defined by the regularity equation together with the
equational relationships $(\forall a,b)(\exists x)\colon a\mathrel{\mathscr{R}}x\mathrel{\mathscr{L}}b$.
This class is however not a variety as, by a theorem of Preston, any
semigroup may be embedded in a regular bisimple monoid \cite[Corollary
1.2.15]{hig}.

The property of $u$ belonging to a subgroup, which we write as $u\in G$,
is also equational: 
\[
(\exists x)\colon (x\in V(u))\wedge(ux=xu).
\]

The ascending chain of the three important classes of $\mathcal{I}$
(\emph{Inverse semigroups}), $\mathcal{O}$ (\emph{Orthodox semigroups}),
and $\mathcal{ES}$, (\emph{Idempotent-solid semigroups}), which are
those regular semigroups whose idempotent generated subsemigroup is a union
of groups, may be defined in a uniform fashion that is conveniently
displayed if we adjoin two redundant equation types to the definition
of regularity: 
\begin{equation}
\text{\textbf{reg}:\,\ensuremath{(\forall\,a,b)(\exists\,x,u,v)\colon (x\in V(a))\wedge(u\in V(a^{2}))\wedge(v\in V(b^{2})).}}\label{eq:reg}
\end{equation}
%%%%%%Adding this:
We include two further classes within this sequence.  For the first, we have the \emph{Right Inverse semigroups} introduced by Venkatesan \cite{ven} as regular semigroups in which each $\mathscr{L}$-class contains a unique idempotent (for that reason, they are also known as \emph{$\mathscr{L}$-unipotent semigroups}).   The class $\mathcal{RI}$  of Right inverse semigroups is given six further characterisations in \cite[Theorem~2.1]{ven}, one of which is the class of all regular semigroups $S$ for which $efe = fe$ for any idempotents $e,f$ in $S$.  It follows that $\mathcal{I} \subseteq  \mathcal{RI} \subseteq \mathcal{O}$. It also is the case that $\mathcal{RI}$ is an EHP class, and the argument for closure under homomorphisms is given in Theorem 3 of \cite{ven}.

The second inclusion in the chain is the class $\mathcal{CN}$ of \emph{Conventional semigroups} of  Masat \cite{mas}: a regular semigroup $S$  is \emph{conventional} if $aea'$ is idempotent for all $(a,a') \in V(S)$ and $e\in E(S)$.  Equivalently, by \cite[Lemma 2.2]{mas}, if $eEe\subseteq E$, which compares as a natural weakening of the $efe = fe$ condition of Right inverse semigroups.  The Conventional semigroup definition is also a weakening of the Orthodox semigroup definition, and so $\mathcal{O} \subseteq \mathcal{CN}$.  A consequence of Masat's Lemma 2.2 and Lallement's Lemma is that $\mathcal{CN}$ is closed under the taking of homomorphisms \cite[Lemma 3.1]{mas}.  Since $\mathcal{CN}$ is clearly closed under direct products and is an elementary class, it follows that $\mathcal{CN}$ is an EHP-class. 
%%%%%%
 With the symbols $a,b,u,v$ satisfying the equations of \textbf{reg},
we may define each these five classes by means of one additional equation (to be included within the scope of the quantification of~$\mathbf{reg}$).  The following proposition is \cite[Theorem 5.2]{higjac}, extended by the inclusion of~$\mathcal{RI}$ and~$\mathcal{CN}$.

\begin{pro}\label{pro:inv} The classes of Inverse
semigroups, Right inverse semigroups, Conventional semigroups, Orthodox semigroups, and Idempotent solid semigroups,
are EHP-classes defined by the following equational bases.
\begin{equation}
\mathcal{I}: \textbf{reg} \wedge aua\cdot bvb=bvb\cdot aua;\label{pro:inv:I}
\end{equation}
\begin{equation}
\mathcal{RI}: \textbf{reg} \wedge aua\cdot bvb\cdot aua = bvb\cdot aua;\label{pro:inv:RI}
\end{equation}
\begin{equation}
\mathcal{O}: \textbf{reg} \wedge aua\cdot bvb\in E;\label{pro:inv:O}
\end{equation}
\begin{equation}
\mathcal{CN}: \textbf{reg} \wedge \ aua\cdot bvb\cdot aua\in E;\label{pro:inv:CN}
\end{equation}
\begin{equation}
\mathcal{ES}: \textbf{reg}\wedge aua\cdot bvb\in G.\label{pro:inv:EH}
\end{equation}
\end{pro}
\begin{proof}
The cases of $\mathcal{I}$, $\mathcal{O}$ and $\mathcal{ES}$ are given in \cite[Theorem 5.2]{higjac}.  The proof for $\mathcal{CN}$ is indicative of the approach, and we omit the very similar argument for $\mathcal{RI}$.  Recall that $\mathbf{reg}$ includes the conditions $u\in V(a^2)$ and $b\in V(b^2)$ so that $aua$ and $bvb$ are idempotents.  Thus any conventional semigroup satisfies the given equation system because of the  $eEe\subseteq E$ condition of \cite[Lemma~2.2]{mas}.   Conversely if $S$ satisfies the equations then $S$ is regular and for any two idempotents $a$ and $b$ we have:
 \[
 aba = a^2b^2a^2 = a^2ua^2\cdot b^2vb^2\cdot a^2ua^2 = aua\cdot bvb\cdot aub,
 \]  
 which is idempotent.
 It follows by \cite[Lemma~2.2]{mas} that we have an equational basis for Conventional semigroups.
\end{proof}
The equational bases in Propositions \ref{pro:reg} and \ref{pro:inv} are not unique. The bases given by equations~\eqref{pro:inv:I},~\eqref{pro:inv:O}, and~\eqref{pro:inv:EH} of Proposition \ref{pro:inv} correspond
to bases for these classes when considered as e-varieties in the sense
of Hall \cite{hal}. These are classes of regular semigroups that are
$HP$-closed, and also closed under the taking of regular subsemigroups.
Indeed any e-variety (of regular semigroups) is an EHP-class of regular
semigroups, and so e-varieties may be defined without the
need to introduce a unary operation that selects arbitrary inverses.
The only semigroup operation involved is the natural operation of
semigroup multiplication. However, not all EHP-classes consisting
of regular semigroups are e-varieties (see \cite[Theorem 5.1]{higjac}). In common
with e-varieties however is the property that if $\mathcal{C}$ is an
EHP-class of regular semigroups then the class $\mathcal{C}^{loc}$ of
all semigroups $S$ whose local subsemigroups $eSe$ lie in~$\mathcal{C}$
($e\in E(S))$ is also an EHP-class (see \cite[Theorem 5.5]{higjac}). 

The abstraction of the idea of e-varieties involves taking an EHP-class
of algebras $\mathcal{N}$, which are labelled \emph{nice}, and then
considering classes $\mathcal{C}$ of nice algebras that are $HP$-closed
and closed under the taking of nice subalgebras. Since the nice algebras
are defined by first order formulae (equation systems), it follows
that $\mathcal{C}$ will be closed under the taking of elementary subalgebras,
and so $\mathcal{C}$ will automatically be another EHP-class. There
could however, as in the case of regular semigroups, be EHP-classes
of nice algebras that were not closed under the taking of nice subalgebras.
If we declare the class of all algebras to be nice, then the corresponding
class of e-varieties coincides with varieties in the usual sense of
algebras defined by identities ($\exists$-free equation systems). 

\subsection{Classes defined by $(\exists\dots)(\forall\dots)$ }\label{subsec:EA}

A simple example of a fundamentally different type is the class $\mathcal{M}$
of \emph{monoids}. 
\begin{equation}
\mathcal{M}:(\exists x)(\forall a)\colon ax=xa=a.\label{eq:monoid}
\end{equation}
We sometimes write these equations as $x=1$, and similarly we write
$x=0$ to abbreviate the equations that ensure the existence of a
zero element in a semigroup. The point to note here however is that
the order of the existential quantifiers in~\eqref{eq:monoid} is $(\exists\,\dots)(\,\forall\,\dots)$,
which is the reverse of all our previous examples. Indeed $\mathcal{M}$
cannot be represented by equations of the type $(\forall\dots)(\exists\dots)$
because any class that has such a basis is closed under the taking
of the union of an ascending chain of algebras from the class, and
the class $\mathcal{M}$ lacks this property. We will investigate this
facet of the theory further in our final section.  

A natural exercise then is to exchange the order of the quantifiers
of the examples of Section \ref{subsec:reg}. This will necessarily result in
a more restricted class to that defined by the original equation system.
Exchanging the order of the quantifiers in equation~\eqref{pro:reg1} defines the class $\mathcal{B}$ of all semigroups that possess a
universal pre-inverse element: 
\begin{equation}
\mathcal{B}:(\exists\,x)(\forall\,a)\colon  a=axa.\label{eq:regular}
\end{equation}
We note that for any $S \in \mathcal{B}$, the $\mathscr{J}$-class $J_x$ is maximal.
Conversely any band $B$ with a maximal $\mathscr{J}$-class $J = J_x$ belongs to $\mathcal{B}$.  
To see this observe that since $B$  is a semilattice of rectangular bands, it follows that for any $a \in \mathcal{B}$, 
$ax \in D_a$. Since $D_a$ is a rectangular band, we have $ax = a^2x= a\cdot ax \mathscr{R} a$, whence $ax\cdot a = a$, showing that $B \in \mathcal{B}$. 

The following is a simple reformulation of the condition defined by \eqref{eq:regular}. 

\begin{lem}\label{lem:2.2.1}
A semigroup $S\in\mathcal{B}$ if and only if $\exists x\in S$
such that $ax\in R_{a}\cap E(S)$ for all $a\in S$, which in turn
is equivalent to the condition that $xa\in L_{a}\cap E(S)$ for all
$a\in S$. 
\end{lem}

\begin{pro}\label{pro:2.2.2}
Suppose that $S\in\mathcal{B}$, let $E$ denote
$E(S)$, and let $x$ denote a fixed choice for satisfying~\eqref{eq:regular}. Then
\begin{enumerate}
\item[(i)] $x\in E$, $S$ satisfies the identity $a^{2}=a^{3}$, and $S=E^{2}$.

\item[(ii)] In $S$, $\mathscr{D}=\mathscr{J}$, and $\mathscr{H}$ is the equality
relation.

\item[(iii)] Let $J=J_{x}$. Then $J$ is the maximum $\mathscr{J}$-class of $S$,  the principal factor $J\cup\{0\}=S/(S-J)\in\mathcal{B}$,
and 
\begin{equation}
R_{x}\cup L_{x}\subseteq E.\label{eq:13}
\end{equation}
\end{enumerate}
\end{pro}
\begin{proof}
(i) Taking $a=x$ in~\eqref{eq:regular} we get $x=x^{3}$. For any
$a$ we have $a^{2}=a^{2}xa^{2}=a(axa)a=a^{3}$. In particular, $x^{2}=x^{3}=x$,
so that $x\in E$. Furthermore, since $(ax)^{2}=ax$ and $(xa)^{2}=xa$
we have $a=axa=(ax)(xa)$, and so $S=E^{2}$. 

(ii) That $\mathscr{D}=\mathscr{J}$ follows from the satisfaction of $a^{2}=a^{3}$,
as this equality of Green's relations is true in any periodic semigroup;
indeed it is true of any group-bound semigroup (see \cite[Theorem
1.2.20]{clipreII}). 

Let $D$ be any (regular) $\mathscr{D}$-class of $S$. In any subgroup
$G$ of $S$, the equation $a^{2}=a^{3}$ implies that $a=e$, the
identity element of $G$, and so $S$ has trivial subgroups. It follows
that every group $\mathscr{H}$-class, and hence every $\mathscr{H}$-class
of $S$ is trivial, which is to say that $S$ is a combinatorial (i.e.
$\mathscr{H}$-trivial) semigroup. 

(iii) Since ${J}_{a}\leq {J}_{x}$ it follows that $J=J_{x}$
is the maximum $\mathscr{J}$-class of $S$ (where $x$ represents any
solution to~\eqref{eq:regular}). Since EHP-classes are closed under the taking of
homomorphisms, the principal factor $J\cup\{0\}$ also belongs to
$\mathcal{B}$, and any solution to~\eqref{eq:regular} in $S$ is also a solution to~\eqref{eq:regular} 
in $J\cup\{0\}$. Suppose that $x\mathrel{\mathscr{R}}a$ in $S$, and hence
in $S/(S-J)$ also. Then since $x\in E$, it follows that $a=xa$.
But then $a=axa=a^{2}$, so that $a\in E$. Dually if $x\mathrel{\mathscr{L}}a$
then $a$ is also idempotent. We conclude that Condition~\eqref{eq:13} holds.
\end{proof}

\begin{remark}\label{rem:2.2.3} It is possible for a semigroup $S$ to satisfy
all three conditions of Proposition \ref{pro:2.2.2} yet for $S$ not to belong to
$\mathcal{B}$. For example take the six-element semigroup $A_{2}^{1}$,
which is the Rees matrix $0$-simple semigroup with adjoined identity
element 1, given by $\mathcal{M}^{0}[\{e\},2,2;P]^{1}$, where $\{e\}$
is a one-element group and 
\[
P=\begin{bmatrix}e & e\\
e & 0
\end{bmatrix}.
\]
Taking $x=1$, we see that each of (i), (ii), and (iii) is satisfied.
However $A_{2}^{1}\not\in\mathcal{B}$: taking $a=1$ we see that we
must take $x=1$ in order to satisfy the condition of Proposition \ref{pro:2.2.2}.
However with $x=1$, for the single non-idempotent element $a=(e;2,2)$
of $A_{2}^{1}$, we have $ax=xa=a\not\in E(S)$, contrary to Lemma \ref{lem:2.2.1}. 
\end{remark}

We next consider the more restricted equational system $\mathcal{V}\subseteq\mathcal{B}$, consisting of all semigroups that possess a universal inverse
element: 
\begin{equation}
\mathcal{V}:(\exists\,x)(\forall\,a)\colon x\in V(a).\label{eq:bisimple}
\end{equation}

Let $S \in \mathcal{V}$, and so Proposition \ref{pro:2.2.2} applies. In particular $S$
is a regular periodic combinatorial semigroup. Moreover, \eqref{eq:bisimple} implies that
$S$ is bisimple, and since any periodic bisimple semigroup is completely simple (Corollary 2.56 of [3]), we conclude that $S$ is completely simple.  
However, a completely simple combinatorial semigroup is none other than a rectangular band, which certainly satisfies \eqref{eq:bisimple}. Indeed this gives
the following curious formulation of the property of the existence of a universal inverse.

\begin{pro} \label{pro:2.2.4}
For any semigroup $S$ either 

\centerline{$ \cap _{a \in S} V(a) = \varnothing$  or $\cap_{a \in S}  V(a) = S$,} 

\noindent the latter occurring if and only if $S$ is a rectangular band. 
\end{pro}

\subsection{Applying the EHP theorem}

The previous section also serves to introduce a strategy for applying
the EHP theorem. The general approach is to systematically list semigroup
equations and identify the corresponding semigroup classes.  As an
example that leads to a new result, we go by way of the following EHP class.

\begin{pro}\label{pro:2.3.1}
The following are equivalent for a semigroup
$S$: 
\begin{enumerate}
\item[(i)] $S$ satisfies 
\begin{equation}
(\forall a)(\exists x)\colon a=a^{2}x.\label{eq:17}
\end{equation}

\item[(ii)] $S$ satisfies $a\mathrel{\mathscr{R}}a^{2}$ for all $a\in S$. 

\item[(iii)] Each $\mathscr{R}$-class of $S$ is a subsemigroup
of $S$. 
\end{enumerate}
\end{pro}
\begin{proof}
(i) $\Rightarrow$ (ii). It follows immediately from
\eqref{eq:17} that $a\mathrel{\mathscr{R}}a^{2}$. 

(ii) $\Rightarrow$(i). Since $a\mathrel{\mathscr{R}}a^{2}$, there exists $y\in S^1$
such that $a=a^{2}y$.  If $y \in S$ we may take $x = y$ in order satisfy \eqref{eq:17}, otherwise put $x = a$.

(ii) $\Leftrightarrow\,$(iii). Clearly $(\text{iii)\ensuremath{\Rightarrow\text{(ii)}}}$.
For the converse direction take $(a,b)\in\mathscr{R}$ in~$S$. Then
since $a\mathrel{\mathscr{R}}a^{2}$ and $\mathscr{R}$ is a left congruence we get
$a\mathrel{\mathscr{R}}a^{2}\mathrel{\mathscr{R}}ab$, and so $R_{a}$ is a subsemigroup of~$S$. 
\end{proof} 

\begin{cor}\label{cor:2.8}
 The following are equivalent for a semigroup $S$: 
\begin{enumerate}
\item[(i)] $S$ satisfies 
\begin{equation}
(\forall a)(\exists x)\colon a=a^{2}xa.\label{eq:19IV}
\end{equation}

\item[(ii)] $S$ satisfies $a\mathrel{\mathscr{R}}a^{2}$ for all $a\in S$ and each $\mathscr{R}$-class
contains an idempotent.

\item[(iii)] $S$ is regular and each $\mathscr{R}$-class of $S$ is a subsemigroup
of $S$. 
\end{enumerate}
\end{cor}
\begin{proof}
 Since equation \eqref{eq:19IV} clearly implies both regularity and equation \eqref{eq:17}, it follows from Proposition 2.7 that (i) implies (iii), while (iii) certainly implies (ii). 
Finally, given (ii),  for any $a \in S$  there exists there exists $y \in S^1$ such that $a = a^2y$.  Moreover, there exists an idempotent $e \in R_a$, whence there exists $z \in S$  such that $az = a^2yz = e$. Then since $ea = a$ we have that $x = yz \in S$ satisfies $a = a^2xa$.   
\end{proof}

\begin{defn}\label{def:2.3.2}
We (temporarily) define the EHP-classes
of \emph{right regular} and \emph{left regular} semigroups, $\mathcal{C}_{r}$
and $\mathcal{C}_{l}$ respectively by the equation systems:
\[
\mathcal{C}_{r}:(\forall a)(\exists x)\colon a=a^{2}xa,\,\,\mathcal{C}_{l}:\,(\forall a)(\exists x)\colon a=axa^{2}.
\]
\end{defn}
The proof of our next theorem was first generated through use of the
software package \emph{Prover~9}~\cite{mcc}.  It is the equational
nature of EHP theory, which is based on the first order language,
that allows computational access to theorems of this nature. 
Although
that proof was not long, some steps seemed unmotivated and so we present a modified
version, where the steps are more natural.  In common with
the Prover 9 argument, our proof  still has two layers of what we might
call nesting, as we choose elements that act on the right to return
squares to their roots. 
%Although
%that proof was not long, some steps seemed unmotivated, for indeed
%they are the result of random testing. The proof here is a modified
%version, where the steps are more natural, although in common with
%the Prover 9 argument, there are still two layers of what we might
%call nesting, as we choose elements that act on the right to return
%squares to their roots. 

\begin{thm} \label{thm:2.3.3}
$\mathcal{C}_{r}=\mathcal{C}_{l}=\mathcal{CR}$.
\end{thm}
\begin{proof}
It follows from Proposition~\ref{pro:2.3.1} and its left-right
dual that any completely regular semigroup is both right regular and
left regular. By symmetry, in order to prove the converse it is enough
to deal with the case where $S\in\mathcal{C}_{r}$, and by the left-right
dual of Proposition~\ref{pro:2.3.1}, it is enough to show that $a\mathrel{\mathscr{L}}a^{2}$
for an arbitrary $a\in S$. 

Choose and fix two members $a',a^{R}\in S$ such that $a=aa'a$ and
$a=a^{2}a^{R}$ and put 
\begin{equation}
r=a^{R}a'a.\label{eq:18}
\end{equation}
Then $r$ has three properties relevant to our purpose, 
\begin{equation}
a^{2}r=(a^{2}a^{R})a'a=aa'a=a.\label{eq:19}
\end{equation}
\begin{equation}
(ar)^{2}=(aa^{R}a'a)(aa^{R}a'a)=(aa^{R}a')(a^{2}a^{R})a'a=aa^{R}a'(aa'a)=a(a^{R}a'a)=ar.\label{eq:20}
\end{equation}
Moreover $a^{2}r=a$ implies $a^{2}r^{2}=ar$ and since \eqref{eq:20} shows
that $ar$ is idempotent we infer:
\begin{equation}
a^{2}r^{2}=ara^{2}r^{2}.\label{eq:21}
\end{equation}
We next deduce that $a=ara$ as follows:
\begin{align*}
a&=a^{2}r\,\,(\text{by \eqref{eq:19})}\\
&=a^{2}r^{2}r^{R}\,\,\text{(by definition of \ensuremath{r^{R}})}\\
&=ara^{2}r^{2}r^{R}\,\,\text{(by \eqref{eq:21})}\\
&=ara^{2}r\,\,\text{(by definition of \ensuremath{r^{R}})}\\
&=ara\,\,\text{(by \eqref{eq:19}).}
\end{align*}
Now $ra\mathrel{\mathscr{L}}a^{2}$ as $ra=a^{R}a'a^{2}$ (by \eqref{eq:18}), and from \eqref{eq:19},  $a^{2}ra=a^{2}$.
Since $a=ara$ we then have $a\mathrel{\mathscr{L}}ra\mathrel{\mathscr{L}}a^{2}$, which completes
the proof.
\end{proof}

Reformulating Theorem \ref{thm:2.3.3} in the terms of Corollary  \ref{cor:2.8}  yields the following result which, to our knowledge,
is not in the literature.  

\begin{cor}\label{cor:2.3.4}
The following are equivalent for a regular
semigroup $S$:
\begin{enumerate}
\item[(i)] each $\mathscr{R}$-class is a subsemigroup of $S$;

\item[(ii)] each $\mathscr{L}$-class is a subsemigroup of $S$;

\item[(iii)] $S$ is completely regular. 
\end{enumerate}
\end{cor}

The equivalence of (i) and (ii) in Corollary \ref{cor:2.3.4} does not hold if we drop the assumption that $S$ is
regular: for example the Baer-Levi semigroup~$S$, which is the semigroup
of all injective mappings $\alpha$ on a countably infinite set $X$
such that $X\setminus X\alpha$ is infinite, is right simple, $\mathscr{L}$-trivial
and idempotent-free, so that $\alpha\mathrel{\mathscr{L}}\alpha^{2}$ is never
true \cite[Lemma 8.3 and Ex. 8.1.1]{clipreII}. 

Corollary \ref{cor:2.3.4} may also be proved by a semantic argument that exploits properties of regular semigroups, as we next show.  This complements the syntactic approach of the proof of Theorem \ref{thm:2.3.3}.  However, as stand alone proofs, the lengths of the two arguments are similar. 

\begin{lem}\label{lem:2.3.5}
An inverse semigroup $S$ has the property that
$a\mathrel{\mathscr{R}}a^{2}$ (resp. $a\mathrel{\mathscr{L}}a^{2})$ for all $a\in S$ if and
only if $S$ is a semilattice of groups.
\end{lem}
\begin{proof}
For the forward implication, take $a\in S$. From the
given condition it follows that $a^{-1}=a^{-2}x$ for some $x\in S$.
Then $(a^{-1})^{-1}=(a^{-2}x)^{-1}$, which is to say $a=x^{-1}a^{2}$,
from which it follows that $a\mathrel{\mathscr{L}}a^{2}.$ Since $a\mathrel{\mathscr{R}}a^{2}$
is given we have $a\mathrel{\mathscr{H}}a^{2}$ for all $a\in S$, so that $S$
is both a union of groups and an inverse semigroup, whence $S$ is
a semilattice of groups. The converse implication is immediate. The
dual statement is also immediate by symmetry.
\end{proof}

We show that (i) implies (iii) in Corollary~\ref{cor:2.3.4}. Let $a\in S$ and
consider the principal factor $P=J_{a}\cup\{0\}$. Since $S$ is regular and $P$ is a quotient of the principal ideal of $a$ (also regular),
 we have that $P$ is also regular.  It follows that $R_{a}^{S}=R_{a}^{P}$, and  $R_{a}$ is a subsemigroup
of $P$. Suppose further that $P$ is completely $0$-simple. Then
since $a\mathrel{\mathscr{R}}a^{2}$ it follows that $a\mathrel{\mathscr{H}}a^{2}$. In this
case $J_{a}$ is a completely simple subsemigroup of $S$. 

Otherwise $P$ is a $0$-simple semigroup that is not completely $0$-simple, in which case $J_{a}$
contains a copy of the bicyclic semigroup $I=\langle x,y:xy=e\rangle$
for any choice of idempotent $e\in J_{a}$ \cite[Theorem 2.54]{clipreI} or \cite[Theorem 1.3.6]{hig}. 
But then $I$ is an inverse semigroup whose $\mathscr{R}$-classes
are subsemigroups of $I$, whence by Lemma~\ref{lem:2.3.5}, $I$ is a semilattice
of groups, which contradicts that $I$ is the bicyclic semigroup (which
is infinite, bisimple, and $\mathscr{H}$-trivial). It follows that this
case does not arise and so $S$ is indeed completely regular. 

We may also show that for any $S\in\mathcal{C}_{r}$, if for a given
$a\in S$ we have $a=a^{2}xa$, then $a=axa^{2}$ (with the dual statement
equally valid). This also follows once we have established that $\mathcal{C}_{r}=\mathcal{CR}$
for in those circumstances we have $a\mathrel{\mathscr{L}}axa$ and $R_{a}\geq R_{axa}$,
whence it follows that $axa\mathrel{\mathscr{H}}a$, as $D_{a}$ is a completely
simple semigroup. Moreover $e=axa$ is the identity of the group $H_{a}$
for $e$ is idempotent: $e^{2}=(axa)(axa)=ax(a^{2}xa)=axa=e.$ Therefore
$a=ea=axa^{2}$, as claimed. 

Theorem \ref{thm:2.3.3} also yields another pair of EHP bases for completely
simple semigroups.

\begin{pro}\label{pro:2.3.6}
The following are equivalent for a semigroup
$S$.
\begin{enumerate}
\item[(i)] $S$ satisfies $(\forall a,b)(\exists x)\colon a=abxba$.

\item[(ii)] $S$ satisfies $(\forall a,b)(\exists x)\colon a=abxa$.

\item[(iii)] $S$ satisfies $(\forall a,b)(\exists x)\colon a=axba$.

\item[(iv)] $S$ is completely simple.
\end{enumerate}
\end{pro}
\begin{proof}
We have by Proposition \ref{pro:reg} that (i) and (iv) are
equivalent. Clearly (i) implies (ii) as given $a=abxba$ we have $a=abya$,
where $y=xb$. Similarly (i) implies (iii). By symmetry it is enough
now to prove that (ii) implies (iv). By taking $b=a$ in (ii) we see
that $S$ satisfies $a=a^{2}xa$, whence by Theorem~\ref{thm:2.3.3}, $S$ is
completely regular. It also follows from (ii) that $S$ has only one
$\mathscr{J}$-class, whence $S$ is completely simple.
\end{proof}

\section{The EHP theorem }\label{sec:EHP}

\subsection{Outline of proof of Theorem \ref{thm:EHP}}

The proof of Theorem \ref{thm:EHP} in the forward direction is simple for it
is clear that if a class of algebras $\mathcal{C}$ is defined by an
equation system then this property is preserved under the taking of
homomorphic images and arbitrary direct products. Moreover, such a
class $\mathcal{C}$ is automatically an elementary class as $\mathcal{C}$
is defined in the first order language. 

The converse direction however is a consequence of Lyndon's
positivity theorem, which states that for an elementary class closed
under taking surjective homomorphic images, a sentence is equivalent
to a positive sentence (one free of negations). Thus we may assume
that our class of algebras $\mathcal{C}$ such that $\mathcal{C\subseteq\,}HP(\mathcal{C})$
is the class of models of a set $\Sigma$ of positive sentences. There
is no loss of generality to assume that all quantifiers are at the
front of the positive sentences. The remaining task then is to show
that disjunctions in each $\rho\in\Sigma$ may be removed.

Consider a sentence $\rho\in\Sigma$. We may express the equation
systems of $\rho$ as a finite conjunction of disjunctions, $\bigwedge_{1\leq i\leq m}\gamma_{i}$,
where each $\gamma_{i}$ is a finite disjunction: $\gamma_{i}=\alpha_{i,1}\vee\dots\vee\alpha_{i,r_{i}}$,
and each $\alpha_{i,j}$ is an atomic formula involving some subset
of the full set of parameters and variables of $\rho$. Suppose that
for some $i$, $r_{i}\geq2$ $(1\leq i\leq m)$. We show that the
conjunct $\gamma_{i}$ may be replaced by some $\alpha_{i,j}$ and
the resulting reduced sentence is equivalent to $\rho$ for any elementary
class $\mathcal{C}$ that is closed under $HP$. Repeating this for each
conjunct $\gamma_{i}$ will see us arrive at the desired $\bigvee$-free
sentence. The quantifiers remain unchanged throughout. 

In view of this, we suppress the symbol $i$ and use the corresponding
symbols $r=r_{i}$, $\gamma=\gamma_{i}$, and $\alpha_{j}=\alpha_{i,j}$.
Let $\rho_{j}$ be the result of replacing $\gamma$ in $\rho$ by
$\alpha_{j}$. Note that $\rho_{j}\vdash\rho$ so that the class of
models satisfied by $(\Sigma\cup\rho_{j})\setminus\rho$ is a subclass
of $\mathcal{C}$. We wish to show that for some $j$ the reverse containment
holds. Assume by way of contradiction that this is not the case. Then
for each $j\,(1\leq j\leq r)$ there exists a model $M_{j}\in\mathcal{C}$
such that $\rho_{j}$ fails in $M_{j}$. Put $M=\Pi_{j=1}^{r}M_{j}\in\mathcal{C}$
and so $M\models\gamma$. (This stage of the argument only requires
that $\mathcal{C}$ is closed under the taking of finitary direct products.
However since $\mathcal{C}$ is an elementary class, $\mathcal{C}$ is closed
under the taking of ultraproducts; $\mathcal{C}$ being closed under
finitary direct products and ultraproducts then implies that $\mathcal{C}$
is in fact closed under the taking of arbitrary direct products.)

The nature of the argument may be illustrated in the simplest case
where $\rho$ has only one pair of existential quantifiers, for instance,
let us say that each $\alpha_{j}$ has a single equation:
\[
\alpha_{j}:(\forall\boldsymbol{a}_{j})(\exists\boldsymbol{x}_{j})\colon u_{j}(\boldsymbol{a}_{j},\boldsymbol{x}_{j})=v_{j}(\boldsymbol{a}_{j},\boldsymbol{x}_{j}),
\]
where $\boldsymbol{a_{j}}$ and $\boldsymbol{x_{j}}$ are the respective
vectors of parameters and variables of the equation $u_{j}=v_{j}$.
Since $M_{j}\not\models\alpha_{j}$ it follows that for $M_{j}$,
$\exists\boldsymbol{a}_{j}$ such that $\forall\boldsymbol{x}_{j}$
$u_{j}(\boldsymbol{a}_{j},\boldsymbol{x}_{j})\neq v_{j}(\boldsymbol{a}_{j},\boldsymbol{x}_{j})$.
Since $M\models\gamma$ we may select $\tilde{\boldsymbol{a}}=(\boldsymbol{a}_{1},\dots,\boldsymbol{a}_{j},\dots,\boldsymbol{a}_{r})$
as our parameter choices for $M$ and there exists a corresponding
choice of variables $\tilde{\boldsymbol{x}}=(\boldsymbol{x}_{1},\dots,\boldsymbol{x}_{j},\dots,\boldsymbol{x}_{r})$
such that for some $j,$ $\boldsymbol{u}_{j}(\tilde{\boldsymbol{a}},\tilde{\boldsymbol{x}})=\boldsymbol{v}_{j}(\tilde{\boldsymbol{a}},\tilde{\boldsymbol{x}})$.
However, taking the projection of this last equation onto the $j$th
component then yields the contradiction that $u_{j}(\boldsymbol{a}_{j},\boldsymbol{x}_{j})=v_{j}(\boldsymbol{a}_{j},\boldsymbol{x}_{j})$. 

In general however, an equation system $\rho\in\Sigma$ may have any
finite number of alternations of existential quantifiers. Since satisfaction
for such a sentence is defined recursively on the string of quantifiers,
the above argument needs to be taken by induction on the number of
quantifiers through the stages outlined in the previous discussion.
This technical argument however does not require any additional facet
to the proof strategy presented in the previous paragraph. The complete
argument is given in \cite[Theorem 3.1]{higjac}.

\subsection{The dual variety theorem for semigroups}

An EHP-class $\mathcal{C}$ defined without the use of the $\exists$
quantifier is a variety, and in particular the class is closed under
the taking of subalgebras. (Birkhoff's theorem says that a class of
algebras $\mathcal{C}$ is defined by a countable list of identities
if and only if $\mathcal{C}$ is closed under the operator $HSP$.) On
the other hand, if the class is defined without the use of the $\forall$
symbol then the class is closed under the taking of superalgebras,
meaning that if $A\in\mathcal{C}$ and $A\leq B$, where $B$ is an algebra
in the defining signature of the algebra class under consideration,
then $B\in\mathcal{C}$ also. 

Here we prove the converse for the class of Semigroups: if an EHP-class
$\mathcal{C}$ is closed under the taking of superalgebras it follows
that $\mathcal{C}$ may be defined by equations of the type $(\exists\,x_{1},\dots,x_{n})\colon (\bigwedge_{1\leq i\leq m}u_i(x_{1},\dots,x_{n})=v_i(x_{1},\dots,x_{n}))$.

\begin{defn}\label{def:3.2.1}
An equation system is called \emph{existential
}if it is has no instances of the $\forall$ quantifier. An EHP-class
is called \emph{existential }if it has a basis of existential equation
systems. 
\end{defn}

In the model theory literature what we call here an existential equation
system is known as a \emph{primitive positive }sentence\emph{. }First
let us suppose that $\mathcal{C}$ is an EHP-class that is closed under
the taking of containing algebras. Suppose that $A$ is an algebra containing
a trivial (one-element) subalgebra, $T$. Then as $T\in\mathcal{C}$ we have
that $A\in\mathcal{C}$ by the containment property. It follows that
if our algebra is of a type where every algebra contains a one-element
algebra, such as Monoids or Groups, then $\mathcal{C}$ is the EHP class
of all algebras of the type under consideration. Moreover, in this
context all algebras satisfy all existential equation systems. It
follows that the converse is trivially true as a class closed under
the taking of containing algebras and a class defined by a basis of
existential equation systems are both necessarily equal to the class
of all algebras.

However, within the class of Semigroups, there are (infinite) semigroups
that are idempotent-free, and so the previous observation does not
apply. For example the equation $(\exists x)\colon x=x^{2}$ defines the
EHP-class of all semigroups that contain an idempotent. This is a
proper class of semigroups that is contained in every existentially
defined EHP-class of semigroups.

\begin{thm}\label{thm:existential}
An EHP-class of semigroups $\mathcal{C}$ is existential
if and only if $\mathcal{C}$ is closed under the taking of containing
semigroups. Equivalently, $\mathcal{C}$ is closed under the taking of
codomains of homomorphisms. 
\end{thm}

Before we embark on the main proof we observe the equivalence with the second sentence. Suppose that the EHP-class $\mathcal{C}$ is closed under
the taking of containing semigroups and let $S\in\mathcal{C}$ with $\alpha:S\rightarrow T$
a homomorphism. Since $\mathcal{C}$ is closed under the taking of homomorphisms,
$S\alpha\in\mathcal{C}$, and since $S\alpha\leq T$ it then follows
that $T\in\mathcal{C}$. Conversely suppose that $\mathcal{C}$ is closed
under the taking of codomains of homomorphisms and suppose that $S\leq T$.
We take $\alpha:S\rightarrow T$ to be the identity mapping on $S$
with codomain $T$, whence by the given condition $T\in\mathcal{C}$.

Most of the remainder of the section is devoted to the proof of Theorem \ref{thm:existential}, which is completed after some preliminary lemmas and discussion. 
The main challenge of the proof is facilitating a kind of quantifier elimination, achieved using the free product construction.  For any semigroup $S$, we consider the free product $F*S$  of $S$ with the free semigroup $F=F_A$ on a countably infinite alphabet $A=\{A_1,A_2,\dots\}$.

Consider an arbitrary equation system $\varepsilon$ satisfied by 
$F*S$: a quantified system $p_1=q_1\wedge\dots\wedge p_\ell=q_\ell$.  As $\varepsilon$ is satisfied, for every evaluation of the parameters (universally quantified variables in $\varepsilon$) in $F*S$, we may find witnesses to the existentially quantified variables, with the choice of each witness being made on the basis of prior quantified variables.  As the parameters can be chosen without restriction, we are going to adopt the strategy that each parameter is chosen to be a free generator from the set $A$ that has not appeared within the evaluation of any variable quantified before it: we refer to this as the \emph{free dependency condition}.  Thus if we have $\forall a_1\exists x_1\forall a_2$, we choose $a_1$ to be $A_1$ and if we have chosen the witness $x_1$ to evaluate as $A_2A_1sA_5$ for some $s\in S$, then we will choose $a_2$ to take the value $A_3$ (or any other free generator except for those already in use: $A_1,A_2,A_5$ in the example).  Without loss of generality however, it is clear that we may rename the free generators so that each parameter $a_i$ is assigned the free generator $A_i$: in the example just given we could rename $A_2$ and $A_3$ and choose the witness $A_3A_1sA_5$ for $x_1$ and then choose $a_2\mapsto A_2$.   We refer to this as an instance of a \emph{canonical evaluation of parameters}, and we say that $F*S$ \emph{satisfies $\varepsilon$ under the canonical evaluation of parameters} to mean that witnesses to the existential variables can be made to achieve equality $p_1=q_1,\dots, p_\ell=q_\ell$ under the evaluation (and satisfying the free dependency condition).  
The following lemma holds in any variety within any signature of algebras, replacing ``free semigroup'' with a relatively free algebra in the variety.  We use  $U$ to denote a semigroup instead of $S$ to match  later usage of the lemma.
\begin{lem}\label{lem:FUU}
Let $\varepsilon$ be an equation system with parameters $a_1,\dots, a_p$ and $U$ be a semigroup.  Let $F=F_A$ be the denumerably generated free semigroup with free generators $A=\{A_1,A_2,\dots\}$.  If $F*U$ satisfies $\varepsilon$ under the canonical parameter evaluation, then $U$ satisfies $\varepsilon$.
\end{lem}
\begin{proof}
Every function $\phi:A\to U$ extends to a retraction from $F*U$ onto the subsemigroup $U$.  We have witnesses $X_1,\dots, X_q\in F*U$ for the canonical evaluation of parameters $a_1,\dots,a_p$ as $A_1,\dots,A_p$ satisfying the free dependency condition.  Let $\gamma$ be any evaluation of the parameters of $\varepsilon$ in $U$, and let $\phi:A\to U$ be $\phi(A_i):=\gamma(a_i)$.  We may extend $\phi$ to a retraction onto $U$, and use witnesses $\phi(X_1),\dots,\phi(X_q)$ in $U$ to verify satisfaction of $\varepsilon$.  
\end{proof}
Note that the free dependency condition is used only at the final step of this proof: if $x_i$ is (existentially) quantified prior to (the universally quantified) $a_j$, we should have that the choice of $\phi(X_i)$ can yield satisfaction of $\varepsilon$ for all subsequent evaluations of $a_j$.  But if $X_i$ contained an occurrence of $A_j$, then the value of $\phi(X_i)$ would in general depend on the evaluation $\gamma(a_j)=\phi(A_j)$ of $a_j$.

Every element of $F*S$ may be written uniquely in the form 
$f_{1}s_{1}f_2\dots s_{k-1}f_{k}$, where each $f_i$ is an element of $F_A$, each $s_i$ is an element of $S$ and  $f_1$ and possibly $f_k$ could be empty (though note that if $k=1$, then this expression is simply $f_1$ and then $f_1$ cannot be empty).  We refer to this as the \emph{normal form} for an element of $F*S$.  If $z_1\dots z_p$ is an arbitrary semigroup word  (possibly with repeats in the sequence of letters $z_1,\dots,z_p$), then under any evaluation of the letters $\{z_1,\dots,z_p\}$ into $F*S$, the word $z_1\dots z_p$ gives rise to a product of normal forms
\begin{equation}
\prod_{1\leq i\leq p}(f_{i,1}s_{i,1}f_{i,2}\dots s_{i,k_i-1}f_{i,k_i}).\label{eq:prod2}
\end{equation}
Depending on the value of $k_i$, and on whether $f_{i,1}$ or $f_{i,k_i}$ are empty, the product in \eqref{eq:prod2} may give rise to sequences of consecutive instances of elements of $S$; a maximal block of consecutive elements of $S$ that arises from such a product will be called an \emph{$S$-run}.  So for example, if letters $z_1,z_2,z_3,z_4$ are evaluated in $F*S$ as $A_1s_1$, $s_2A_2s_1A_3s_3$, $s_2$ and $s_1A_1$ respectively, then the word $z_1z_2z_3z_4$ evaluates as $(A_1s_1)(s_2A_2s_1A_3s_3)(s_2)(s_1A_1)$, and we have $S$-runs $s_1s_2$, $s_1$ and $s_3s_2s_1$.  These of course collapse to individual elements of $S$ in the reduction of \eqref{eq:prod2} to normal form, but we are interested in the uncollapsed form.  For each $S$-run we may also create an \emph{abstract run}, which is a matching semigroup word in the alphabet $\{x_s\mid s\in S\}$ of variables  indexed by elements of $S$; so the $S$-run $s_3s_2s_1$ becomes $x_{s_3}x_{s_2}x_{s_1}$.

If $F*S$ satisfies an equation system $\varepsilon$ under the canonical evaluation of parameters, then the chosen witnesses in $F*S$ will provide a collection of equalities between $S$-runs.  More precisely, each equality in $\varepsilon$ produces an equality in $F*S$ of the form
\begin{equation}
f_{1}S_{1}f_2\dots S_{k_i-1}f_{k_i} = g_{1}T_{1}g_2\dots T_{\ell_i-1}g_{\ell_i}
\end{equation}
where $S_1,\dots, S_{k_i-1}$ and $T_1,\dots,T_{\ell_i-1}$ are $S$-runs.  As this is an equality holding in the free product, it follows that $k_i=\ell_i$ and that $S_1=T_1, \dots, S_{k_i-1}=T_{k_i-1}$ in $S$ (and that $f_i=g_i$ are identical as words in $A^+$, or empty).   For such a choice $c$ of witnesses to the canonical evaluation, let $\varepsilon_c$ denote the existential equation systems consisting of the conjunction of the equalities arising from the resulting abstract $S$-runs, across the witnessing evaluations of all the equalities $p_1=q_1\wedge\dots\wedge p_\ell=q_\ell$.    As an example, consider the equation system 
\begin{multline*}
\varepsilon:(\forall a_1)(\exists x_1)( \forall a_2)(\exists x_2\exists x_3)\colon  (a_1a_2x_1x_2x_3a_1=a_1x_2a_1a_2x_3x_1\\
\wedge a_1x_3a_1=a_1x_3x_3a_1),
\end{multline*}
and consider a semigroup $S$ in which the canonical evaluation of $a_1,a_2$ into $F*S$ has choices for $x_1,x_2,x_3$ satisfying the free dependency condition that lead to satisfaction of the equalities: an example might be $X_1:=s_1A_1$, $X_2:=A_2s_1$ and $X_3=s_2$.  (Note that as $x_1$ is quantified prior to $a_2$ the free dependency condition requires that the word $X_1$ should not involve $A_2$, because the value of $x_1$ should not in general depend on the choice of $a_2$.)  In this hypothetical scenario, we have 
\[
(A_1)(A_2)(s_1A_1)(A_2s_1)(s_2)(A_1)=(A_1)(A_2s_1)(A_1)(A_2)(s_2)(s_1A_1)
\]
 and 
 \[
 (A_1)(s_2)(A_1)=(A_1)(s_2)(s_2)(A_1).
 \] 
 From the first equality we find that $s_1=s_1$ and $s_1s_2=s_2s_1$.  From the second equality we find that $s_2=s_2s_2$.  Then for this choice $c$ we obtain $\varepsilon_c$ as 
 \[
 (\exists x_{s_1}\exists x_{s_2})\colon  (x_{s_1}=x_{s_1}\wedge x_{s_1}x_{s_2}=x_{s_2}x_{s_1}\wedge x_{s_2}=x_{s_2}x_{s_2}).
 \]  
Obviously the equality $x_{s_1}=x_{s_1}$ here is redundant and could be removed.  One can further see that $s_1$ could be replaced by $s_2$ in this sentence, so that $\varepsilon_c$ is logically equivalent to $(\exists x)\ x=x^2$.
\begin{lem}\label{lem:ece}
Let $\varepsilon$ be an equation system satisfied by $F*S$, under the canonical parameter evaluation for some semigroup $S$.  Then for any choice $c$ of witnesses in $F*S$ to the satisfaction of the equalities in $\varepsilon$ we have $S\models \varepsilon_c$ and $\varepsilon_c\vdash \varepsilon$ within the class of semigroups.
\end{lem}
\begin{proof}
Fix a choice $c$ of witnesses, $x_i\mapsto X_i\in F*S$.  The construction of $\varepsilon_c$ from this choice trivially ensures that $S\models \varepsilon_c$, which proves the first claim.  Now assume that $U$ is any semigroup satisfying $\varepsilon_c$; we show that $U\models \varepsilon$.  We may find witnesses to the canonical evaluation of parameters in $F*U$, as the witnesses to satisfaction of $\varepsilon_c$ in $U$ provide the required $U$-runs to match those from $S$ in the satisfaction of $\varepsilon$ in $F*S$ under the canonical parameter evaluation: note that these will not violate the free dependency condition for existential witnesses, as they only provide the required $U$-runs, with any free parameters chosen in the evaluation of a variable $x_i$ following whatever was done in $S$ for the choice $c$ (which satisfied the free dependency condition).  Then Lemma \ref{lem:FUU} shows that~$U$ satisfies $\varepsilon$.  
\end{proof}
Referring to the example given prior to Lemma \ref{lem:ece}, where $\varepsilon_c$ was logically equivalent to $(\exists x)\ x=x^2$, the statement $\varepsilon_c\vdash \varepsilon$ in Lemma \ref{lem:ece} is asserting that any semigroup containing an idempotent will satisfy 
\[
(\forall a_1)(\exists x_1)(\forall a_2)(\exists x_2\exists x_3)\colon  (a_1a_2x_1x_2x_3a_1=a_1x_2a_1a_2x_3x_1\wedge a_1x_3a_1=a_1x_3x_3a_1).
\]

For a given equation system $\varepsilon$ and any $S$ for which $F*S$ satisfies $\varepsilon$ under the canonical parameter evaluation, the maximal length  of any $S$-run is bounded by the maximal number of consecutively adjacent existential variables within a word occurring in $\varepsilon$.  Thus there is an upper bound $\ell$ on the number of $S$-runs of length more than $1$ that appear in $\varepsilon_c$ (independent of $S$ and $c$).  Because equalities between $S$-runs of length~$1$ are trivial (they yield equalities $x_s=x_s$ for some $s\in S$), the number of nontrivial conjuncts within $\varepsilon_c$ is at most $\ell$ also, so that up to logical equivalence, there are only finitely many different existential equation systems of the form $\varepsilon_c$, with the number determined by the structure of the sentence $\varepsilon$ only.  We let $H(\varepsilon)$ denote this (finite) set of existential sentences.  The case of $H(\varepsilon)=\varnothing$ is possible, and corresponds to the situation where there are no nontrivial $S$-runs in any canonical evaluation, for any $S$.  In this situation, note that $\varepsilon$ is trivially equivalent to the equation system $\varepsilon'$ obtained by including in $\varepsilon$ the conjunct $x=x$, where $x$ is a new (existentially quantified) variable.  Because any element of $F*S$ will satisfy $x=x$, we have the same witnesses as previously for $\varepsilon$ (which by assumption all avoided any $S$-runs), along with an arbitrary witness for $x$.  Thus $H(\varepsilon')=\{(\exists x)\colon x=x\}$, so that we may let $\varepsilon_c$ denote the equation system $(\exists x)\colon x=x$.

\begin{proof}[Proof of Theorem \ref{thm:existential}.]
Clearly an existential class is closed under the taking of containing
semigroups. Conversely let us suppose that $\mathcal{C}$ is an EHP-class
of semigroups that is closed under the operation of taking containing
semigroups.  

Let $\mathcal{B}$ be the existential class with EHP basis $E'$, which
is the set of all existential equation systems that are satisfied
by all members of $\mathcal{C}$. Clearly $\mathcal{C}\subseteq\mathcal{B}$,
and indeed~$\mathcal{B}$ is the smallest existential class that contains
$\mathcal{C}$. Our task is to prove the reverse containment. 

Let $E$ be a set of equation systems characterising $\mathcal{C}$ and consider any $\varepsilon\in E$.  
We will show that there is an equation of the form $\varepsilon_c$ that can replace $\varepsilon$.  
Repeated application, across all members of $E$ leads to a subset of $E'$, which will complete the proof.

As a first step we need to show that there is an equation of the form $\varepsilon_c$ that holds on all members of~$\mathcal{C}$.  We may assume that $H(\varepsilon)$ is not empty.  Assume for contradiction that for each $\varepsilon_c\in H(\varepsilon)$ there is $S_c\in \mathcal{C}$ that fails~$\varepsilon_c$.  As $\mathcal{C}$ is closed under taking direct products we have that $T:=\prod_{\varepsilon_c\in H(\varepsilon)}S_c\in \mathcal{C}$ and then as $\mathcal{C}$ is closed under taking containing semigroups we have that $F*T\models \varepsilon$.  But then $T\models \varepsilon_c$ for some choice of witnesses to the canonical parameter evaluation.  But then all quotients of $T$ satisfy $\varepsilon_c$, contradicting the assumption that $S_c$ fails~$\varepsilon_c$.

%Now assume that $H(\varepsilon)$ is empty.  This is somewhat of a degenerate case, as it is apparent that $\varepsilon$ can be removed entirely from $E$, or equivalently replaced by the tautology $(\exists x)\colon x=x$, which trivially holds in all semigroups.  One way to see this is to observe that we may replace $\varepsilon$ by a new equation system by introducing a new existentially quantified variable $x$, not appearing elsewhere in $\varepsilon$, and adding the conjunct $x=x$.  This new equation system $\varepsilon'$ holds equivalently to $\varepsilon$, and $H(\varepsilon')=\{(\exists x)\colon x=x\}$, so that we may choose $\varepsilon_c$ to be $(\exists x)\colon x=x$. 

Thus for every $\varepsilon\in E$ there is an existential equation system of the form $\varepsilon_c$ such that $\varepsilon_c\in E'$.  By Lemma \ref{lem:ece} we have that $\varepsilon_c\vdash \varepsilon$, so that $\varepsilon_c$ can replace $\varepsilon$ in $E$.
\end{proof}

The existential equations that are
satisfied by every semigroup are explicitly identified in \cite[Corollary
6.7]{higjac}.
To conclude this section we note that this may be extended, at least at the level of an algorithmic solution to arbitrary equation systems.  Satisfying equations in $F*S$ by canonical parameter evaluation is somewhat reminiscent of solving equations in free semigroups, a problem originally solved by Makanin \cite{mak}, and of continued interest and development.  The connection turns out to be genuine, and a strong form of Makanin's algorithm can be used to show that it is decidable to determine if an equation system holds in the variety of all semigroups.  
\begin{thm}\label{thm:makanin}
The class of equation systems satisfied in the class of all semigroups is decidable.
\end{thm}
\begin{proof}
As we now explain, Theorem \ref{thm:makanin} is a direct corollary of Makanin's celebrated solution \cite{mak} to solvability of equations over free semigroups, as extended to allow for rational constraints (see Chapter 12 of \cite{lot}).
 In the context of Makanin's algorithm, a system of equations on the free semigroup consists of a finite set of equalities $\varepsilon$ between words in alphabet $\{A_1,A_2,\dots\}\cup\{X_1,X_2,\dots\}$ and we are asked whether there is a satisfying evaluation of the variables $X_i$ in the free semigroup $A^+$ (where $A=\{A_1,A_2,\dots\}$ as before).  In the absence of any constraint on the choice of the $X_i$, this coincides  with what we called the canonical parameter evaluation of the generators $A_1,A_2,\dots$ (as themselves) in the equation system $(\forall A_1,\dots,A_n)(\exists X_1,\dots,X_m)\varepsilon$ (for suitable $n,m$ determined by the variables that appear in $\varepsilon$).    Lothaire \cite[\S12.1.8]{lot} details an extension of Makanin's algorithm to allow for the variables $X_1,X_2,\dots$ to be constrained by rational languages $\lambda_1,\lambda_2,\dots$ over the alphabet $A$.  This enables us to additionally enforce the free dependency condition for an equation system~$\varepsilon$: constrain each variable~$X_i$  to lie in the rational language $\lambda_i$  excluding letters in~$A$ that are quantified in $\varepsilon$ to the right of the existential quantification of $X_i$.  Satisfaction of this constrained instance of the equation problem coincides in definition precisely with satisfaction of $\varepsilon$ under the canonical parameter evaluation with free dependency condition holding.  But this latter property is equivalent to unconditional satisfaction of the equation system $\varepsilon$: for the nontrivial direction, use the universal mapping property of $A^+$ with respect to itself (or alternatively use Lemma \ref{lem:ece}, using the fact that $A^+*A^+\cong A^+$).  Thus the extension of Makanin's algorithm in \cite[\S12.1.8]{lot} can be used to decide satisfaction of arbitrary equation systems on $A^+$.
 
 Finally, the equation systems true on $A^+$ (with $A$ a denumerable alphabet) are precisely those in the class of all semigroups.  One direction of this claim follows trivially from the fact that $A^+$ is a semigroup.  For the other direction, use the universal mapping property for $A^+$ with respect to any other semigroup $S$ to find that if $\varepsilon$ is satisfied by $A^+$ (under canonical parameter evaluation, with free dependency), then it is satisfied by $S$.
\end{proof}

\section{EHP-classes requiring both types of quantifier alternation }

An open question in this new theory of EHP-classes is: For any $n$,
does there exist an EHP-class $\mathcal{C}$ such that in any basis for
$\mathcal{C}$, $n$ alternations of the existential quantifiers $\forall$
and $\exists$ are \emph{necessary }in at least one equation of the
basis? 

Example 2.3(iii) of \cite{higjac} is the EHP-class $\mathcal{C}$ defined
by:
\[
(\exists y)(\forall a)(\exists x,z)\colon  a=xyz,
\]
which is the class of semigroups $S$ that have a maximum $\mathscr{J}$-class
$J$ such that $S/(S-J)$ is not a null semigroup. It is proved there
that $\mathcal{C}$ cannot be defined by equations systems exclusively of the type $(\forall \cdots )(\exists \cdots)$  nor by systems with the reverse order of quantifiers $(\exists \cdots)(\forall \cdots)$.  

The strategy for proving this type of result is two-fold. To show
that $(\forall\dots)(\exists\dots)$ quantification is not possible
we find a chain of semigroups $S_{1}\subseteq S_{2}\subseteq\dots$,
with each $S_{n}\in\mathcal{C}$, such that the semigroup union $S=\cup_{n=1}^{\infty}S_{n}\not\in\mathcal{C}$.
It follows from this that $\mathcal{C}$ cannot be captured by an equation
system based on a $(\forall\dots)(\exists\dots)$ quantification
as that would imply that $S\in\mathcal{C}$ as well. (This phenomenon
we have already observed in Section~\ref{sec:classical} in the context of the EHP-class
of Monoids.)

Next we wish to show that definition by a $(\exists\dots)(\forall\dots)$
quantification for $\mathcal{C}$ is also impossible. This is done by
identifying a semigroup chain $S_{1}\subseteq S_{2}\subseteq\dots$,
where no member of the chain lies in $\mathcal{C}$, yet their union
$S=\cup_{n=1}^{\infty}S_{n}$ is a semigroup in $\mathcal{C}$. Given
that $S\in\mathcal{C}$, if $\mathcal{C}$ possessed a quantification of
the form $(\exists\dots)(\forall\dots)$, it would follow that $S_{n}\in\mathcal{C}$
for all sufficiently large $n$. 

Our Example \ref{eg:4.1} is complementary to the previous one in that it is
defined by quantification of the type $(\forall\dots)(\exists\dots)(\forall\dots)$.

We are working with the class of Semigroups throughout and all quantifications
are assumed to take place in some arbitrary semigroup $S$. Letters
at the front (resp. end) of the alphabet $a,b,c,$ (resp. $x,y,z$)
denote parameters (resp. variables) in a given equation, typically
written $e:p=q$.

\begin{example}\label{eg:4.1}
Let $\mathcal{C}$ be the EHP-class defined by 
\begin{equation}
\mathcal{C}:\,(\forall a)(\exists x)(\forall b)\colon axb=abx.\label{eq:22}
\end{equation}
Then $\mathcal{C}$ cannot be defined by equation systems of the type $(\forall \cdots )(\exists \cdots)$  nor by systems with the reverse order of quantifiers $(\exists \cdots)(\forall \cdots)$.  
\end{example}
\begin{proof}
We prove the claim by applying the strategy we have just outlined. 

Let $F_{A}$ be the free semigroup on the infinite set of generators
\[
A=\{a_{1},a_{2},\dots,x_{1},x_{2},\dots\}.
\]
Denote the finite subset $\{a_{1},\dots,a_{n},x_{1},\dots,x_{n}\}$
of $A$ by $A_{n}$. Let $\rho$ be the congruence on $F_{A}$ generated
by
\[
\rho^{0}=\{(b_{i}x_{j},x_{j}b_{i}):\,\text{where\,\ensuremath{\,b_{i}\in\{a_{i},x_{i}\},\,}}i\leq j\},
\]
and put $S=F_{A}/\rho$. Let $S_{n}\leq S$, where $S_{n}=\langle a_{1}\rho,\dots,a_{n}\rho,x_{1}\rho,\dots,x_{n}\rho\rangle$.
In the following argument we suppress the symbol $\rho$ and write
$a\in S_{n},b\in S$ and such like to stand for $a\rho\in S_{n},b\rho\in S$. 

Thus we have a semigroup chain $S_{1}\leq S_{2}\leq\dots\leq S_{n}\leq\dots\leq S$.
Observe that $S_{n}\in\mathcal{C}$ as for any $a,b\in S_{n}$ we put
$x=x_{n}$ and note that $axb=abx$ as $x_{n}$ commutes with each
member of $A_{n}$. However we now show that $S\not\in\mathcal{C}$.
Put $a=a_{1}$ and let $x\in S$. Then $x\in S_{n}$ for some $n$.
Put $b=a_{n+1}$ and consider $axb=a_{1}xa_{n+1}$. Since $a_{n+1}$
commutes only with $x_{j}$ where $j\geq n+1$, we infer that any
word $w\in F_{A}$ such that $a_{1}xa_{n+1}=w$ in $S$ has the form
$w=w_{1}a_{n+1}$ where $w_{1}$ is the result of a permutation of
the letters of $a_{1}x$. In particular $a_{1}xa_{n+1}\neq a_{1}a_{n+1}x$
in $S$. Hence $S$ does not satisfy the equation system \eqref{eq:22}. Therefore
$\mathcal{C}$ cannot be defined by equations using only $(\forall\dots)(\exists\dots)$
quantification. 

Next we wish to show that definition by a $(\exists\dots)(\forall\dots)$
quantification for $\mathcal{C}$ is also impossible. This is done by
identifying a semigroup chain $S_{1}\leq S_{2}\leq \cdots$,
where no member of the chain lies in $\mathcal{C}$, yet their union
$S=\cup_{n=1}^{\infty}S_{n}$ is a semigroup in $\mathcal{C}$. The existence
of such a chain together with the chain of the previous paragraph establishes our claim. 

Let $F_{n}$ be the free semigroup on $A_{n}=\{a_{1},\dots,a_{n},x_{1},\dots,x_{n}\}$
$(n\geq1)$ and let $F_{A}$ denote the free semigroup on $A$$=\cup_{n=1}^{\infty}A_{n}$.
Let $\rho$ be the congruence on $F_{A}$ generated by $\rho^{0}$
where
\[
\rho^{0}=\{(ax_{n+1}b,\,abx_{n+1}),\,n\geq1,\,a\in F_{n},\,b\in F_{A}\}
\]
and put $S=F_{A}/\rho$. Then $S\in\mathcal{C}$ for let us take any
$a\rho\in S$. Then $a\in F_{n}$ for some $n\geq1$. Take $x=x_{n+1}$
and let $b\rho\in S$. Then $(ax_{n+1}b,\,abx_{n+1})\in\rho^{0}$,
so that in $S$ we have $axb=abx$, and therefore $S\in\mathcal{C}$. 

Now define $S_{n}=\{a\rho:a\in F_{n}\}$. Then $S_{n}\leq S$ and
we have the semigroup chain $S_{1}\leq S_{2}\leq \cdots\leq  S_{n}\leq S_{n+1}\leq\cdots$,
with $S$ being the union of this chain. Now take $a=a_{n}$ and consider
$a\rho$, $x\rho,b\rho\in S_{n}$. Then $a\rho x\rho b\rho=(a_{n}xb)\rho$.
Since $a_{n}\not\in F_{m}$ for any $m<n$, the only $\rho^{0}$ pairs
involving a word with initial letter $a_{n}$ contain some letter
$x_{n+m}\not\in A_{n}$ $(m\geq1)$. Hence $a_{n}xb\neq a_{n}bx$
in $S_{n}$, for any $b$ that is not a power of $x$, from which
we infer that $S_{n}$ does not satisfy the equation system \eqref{eq:22}.
We conclude that $S_{n}\not\in\mathcal{C}$ for all $n\geq1$ but, as
we have witnessed, $S\in\mathcal{C}$. 

We therefore conclude that $\mathcal{C}$ cannot be
defined by a basis of equation systems of the type $(\forall \cdots )(\exists \cdots)$  nor of the type $(\exists \cdots)(\forall \cdots)$.  
\end{proof}

ACKNOWLEDGEMENTS We would like to thank both referees for their constructive comments, in particular those concerning Propositions \ref{pro:2.2.2} and \ref{pro:2.2.4}.


\begin{thebibliography}{1}
\bibitem{ber} Bergman, C., Universal Algebra: Fundamentals and Selected Topics, Chapman Hall/CRC Press, 2011.

\bibitem{bursan} Burris, S. and Sankappanavar, H.P., A Course in Universal
Algebra, Springer-Verlag, 1981.

\bibitem{clipreI} Clifford A.H. and Preston, G.B., The algebraic theory
of semigroups, Vol. I, AMS Mathematical Surveys, (1961).

\bibitem{clipreII} Clifford A.H. and Preston, G.B., The algebraic theory
of semigroups, Vol. II, AMS Mathematical Surveys, (1967).

\bibitem{hal} Hall, T.E., Identities for existence varieties of
regular semigroups, Bull. Austral. Math. Soc. 40 (1989), 59--77. 

\bibitem{hig} Higgins, P.M., Techniques of semigroup theory, OUP
(1992).

\bibitem{higjac} Higgins, P.M., and M. Jackson, Algebras defined by
equations, J. Algebra, 555, 131--156, (2020).

\bibitem{how} Howie, J.M., Fundamentals of Semigroup Theory, London
Math. Soc. monographs, Clarendon Press, OUP (1995).

\bibitem{lot} Lothaire, M., Algebraic combinatorics on words,  Encyclopedia of Mathematics and its Applications, 90. Cambridge University Press, Cambridge, 2002. xiv+504 pp.

\bibitem{mcc} McCune, W., Prover 9 and Mace 4, http//:www.cs.unm/\textasciitilde mccune/Prover9,
2005-2010.

\bibitem{mak} Makanin, G. S.,
The problem of the solvability of equations in a free semigroup. (Russian) 
Mat. Sb. (N.S.) 103(145) (1977), no. 2, 147--236, 319. 

\bibitem{mas}   Masat, F.E., Congruences on conventional semigroups, Czech. Math. J. 31(106) (1981), 199--205. 

\bibitem{ven} Venkatesan, P.S., Right (Left) inverse semigroups, J. Algebra 31 (1974), 209--217.
\end{thebibliography}
\end{document}